\documentclass[a4paper,12pt,reqno]{amsart}
\usepackage{graphics}
\usepackage{amsmath,amssymb,amsthm,graphicx,verbatim,psfrag,multirow}
\usepackage[colorlinks=true]{hyperref}
\usepackage{color}
\usepackage{ulem}
\graphicspath{{figs/}}

\theoremstyle{plain}
\newtheorem{theo}{Theorem}
\newtheorem{lem}[theo]{Lemma}

\newtheorem{cor}[theo]{Corollary}
\newtheorem{prop}[theo]{Proposition}
\newcommand{\inv}{\operatorname{inv}}

\theoremstyle{definition}
\newtheorem{defi}{Definition}[section]

\newcommand{\e}{\operatorname{E}}
\newcommand{\fd}{\overline{\Delta}}
\newcommand{\bd}{\underline{\Delta}}
\newcommand{\id}{\operatorname{Id}}

\newcommand{\ct}{\operatorname{CT}}
\newcommand{\op}{\operatorname{OP}}
\newcommand{\sgn}{\operatorname{sgn}}

\newcommand{\sym}{{\operatorname{\mathbf{Sym}}}}
\newcommand{\asym}{{\operatorname{\mathbf{ASym}}}}
\renewcommand{\ast}{\operatorname{AST}}
\newcommand{\tast}{\operatorname{\widetilde{AST}}}

\newcommand{\ext}{\operatorname{ext}}
\renewcommand{\path}{\operatorname{Path}}
\newcommand{\lp}{\operatorname{LP}}

\numberwithin{equation}{section}

\begin{document}

\title[Alternating sign triangles]{Enumeration of alternating sign triangles using a constant term approach}

\author[Ilse Fischer]{Ilse Fischer}
\address{Ilse Fischer, Fakult\"{a}t f\"{u}r Mathematik, Universit\"{a}t Wien, Oskar-Morgenstern-Platz 1, 1090 Wien, Austria}
\email{ilse.fischer@univie.ac.at}

\thanks{The author acknowledges support from the Austrian Science Foundation FWF, START grant Y463 and SFB grant F50.}


\begin{abstract}
Alternating sign triangles (ASTs) have recently been introduced by Ayyer, Behrend and the author, and it was proven that there is the same number of ASTs with $n$ rows as there is of $n \times n$ alternating sign matrices (ASMs). We prove a conjecture by Behrend on a refined enumeration of ASTs with respect to a statistic that is shown to have the same distribution as the column of the unique $1$ in the top row of an ASM. The proof of the conjecture is based on a certain multivariate generating function of ASTs that takes the positions of the columns with sum $1$ ($1$-columns) into account. We also prove a curious identity on the cyclic rotation of the $1$-columns of ASTs. Furthermore, we discuss a relation of our multivariate generating function to a formula of Di Francesco and Zinn-Justin for the number of fully packed loop configurations associated with a given link pattern. The proofs of our results employ the author's operator formula for the number of monotone triangles with prescribed bottom row. This is opposed to the six-vertex model approach that was used by Ayyer, Behrend and the author to enumerate ASTs, and since the refined enumeration implies the unrefined enumeration, the present paper also provides an alternative proof of the enumeration of ASTs.  
\end{abstract}

\maketitle

\section{Introduction}

An \emph{alternating sign triangle} (AST) of order $n$ is a triangular array 
of the following form
$$
\begin{array}{cccccccc}
a_{1,1}  & a_{1,2} & a_{1,3} & \ldots & \ldots &  \ldots & a_{1,2n-1} \\
& a_{2,2} & a_{2,3} &\ldots & \ldots & a_{2,2n-2} & \\
& & \ldots & \ldots & \ldots & & \\
&&& a_{n,n} &&&
\end{array}
$$
with $a_{i,j} \in \{0, 1, -1\}$ such that the following conditions are satisfied.
\begin{enumerate}
\item\label{AST-1} The non-zero entries alternate in each row and each column.
\item\label{AST-2} All row sums are $1$.
\item\label{AST-3} The topmost non-zero entry of each column is $1$ (if there exists any such entry).
\end{enumerate}
ASTs have recently been introduced in \cite{EXTREMEDASASM}.

The $7$ ASTs of order $3$ are the following:
\begin{gather*}
\begin{array}{ccccc} 1 & 0 & 0 & 0 & 0 \\ & 1 & 0 & 0 & \\ && 1 &&
\end{array}, \quad
\begin{array}{ccccc} 1 & 0 & 0 & 0 & 0 \\ & 0 & 0 & 1 & \\ && 1 &&
\end{array}, \quad
\begin{array}{ccccc} 0 & 1 & 0 & 0 & 0 \\ & 0 & 0 & 1 & \\ && 1 &&
\end{array}, \quad
\begin{array}{ccccc} 0 & 0 & 0 & 1 & 0 \\ & 1 & 0 & 0 & \\ && 1 &&
\end{array}, \\
\begin{array}{ccccc} 0 & 0 & 0 & 0 & 1 \\ & 1 & 0 & 0 & \\ && 1 &&
\end{array}, \quad
\begin{array}{ccccc} 0 & 0 & 0 & 0 & 1 \\ & 0 & 0 & 1 & \\ && 1 &&
\end{array}, \quad
\begin{array}{ccccc} 0 & 0 & 1 & 0 & 0 \\ & 1 & -1 & 1 & \\ && 1 &&
\end{array}
\end{gather*}
\emph{Alternating sign matrices} (ASMs), on the other hand, were introduced more than $35$ years ago \cite{RR86,MilRobRum82,BressoudPandC} and are defined as square matrices with entries in $\{0,1,-1\}$ such that, along each row and each column, the non-zero elements alternate and add up to $1$.
It is no coincidence that there are also $7$ ASMs of order $3$ as it was shown  in \cite{EXTREMEDASASM} that, for any positive integer $n$, there is the same number of order $n$ ASMs as there is of order $n$ ASTs. Since the proof is not bijective, this result gives rise to the search for an \emph{explicit bijection} between ASTs and ASMs of a given order as well as between ASTs and the two other objects that are known to be enumerated by the same numbers, that is  \emph{totally symmetric self-complementary plane partitions} (TSSCPPs) in an $2n \times 2n \times 2n$ box \cite{MilRobRum86} and 
\emph{descending plane partitions} (DPPs) whose parts do not exceed $n$ \cite{DPPMRR}. Presumably, this task is no simpler than constructing explicit bijections between ASMs and TSSCPPs as well as ASMs and DPPs which remains to be elusive although many combinatorialists have tried in the past $35$ years (see \cite{KrattGogMagog,FonZinn, cheballah1, cheballah2, betti} and
\cite{DPP1Behrend,DPP2Behrend} 
for contributions concerning the relation between ASMs and TSSCPPs, respectively between ASMs and DPPs). 

The approach that was used in \cite{EXTREMEDASASM} to enumerate ASTs of a given order employs the six-vertex model. In the present paper, we use an alternative approach (which was developed in 
\cite{FischerNumberOfMT,FischerNewRefProof,FischerSimplifiedProof,FischerRefEnumASM,FischerASMProof2015})
to give another proof of this result and to provide refinements that currently seem to be inaccessible using the six-vertex model approach.  

The refined alternating sign matrix theorem \cite{ZeilbergerRefinedASMProof} provides a simple product formula for the number of $n \times n$ ASMs with prescribed position of the unique $1$ in the top row. Behrend formulated a conjecture \cite[Section~5.6]{EXTREMEDASASM} on a remarkable statistic on ASTs for which he claims that it has the same distribution as the position of the unique $1$ in the top row of an ASM. Next we define this statistic.

In an AST of order $n$, the total sum of entries is $n$ as all $n$ rows sum to $1$. Since every column has sum $0$ or $1$, there are precisely $n$ columns that have sum $1$. In the following, we refer to such columns as \emph{$1$-columns}. Note that the central column is always a $1$-column. Among the $1$-columns of an AST, we can distinguish between those whose bottom entry is $0$ or $1$. The first are referred to as $10$-columns, while the second are referred to as $11$-columns. For an AST $T$, the following statistic $\rho$ is of interest.
\begin{multline*}
\rho(T) = \text{\# of $11$-columns left of the central column} \\
+\text{\# of $10$-columns right of the central column}+1
\end{multline*}
For the ASTs of order $3$ in our list above, these statistics are $(3,2,1,3,2,1,2)$ in this order.
One of our main results is the proof of the next theorem.

\begin{theo} 
\label{rogerrefined}
Let $n,r$ be integers such that $1 \le r \le n$. There is the same number of $n \times n$ ASMs $A=(a_{i,j})_{1 \le i, j \le n}$ with $a_{1,r}=1$ as there is of order $n$ ASTs $T$ with $\rho(T)=r$.
\end{theo}

The proof of this result is based on a slight generalization (see Theorem~\ref{constantterm1}) of the following result on the refined enumeration of ASTs w.r.t.\ the positions of the $1$-columns, which is of interest in its own right.

\begin{theo} 
\label{main}
Let $n$ be a positive integer and $0 \le j_1 < j_2 < \ldots < j_{n-1} \le 2n-3$. The number $\ast(n;j_1,\ldots,j_{n-1})$ of order $n$ ASTs that have the $1$-columns in positions $j_1,j_2,\ldots,j_{n-1}$, where we count from the left starting with $0$ and disregard the central column, is the coefficient of $X_1^{j_1} X_2^{j_2} \ldots X_{n-1}^{j_{n-1}}$ in 
\begin{equation}
\label{gfun}
\prod_{i=1}^{n-1} (1+X_i) \prod_{1 \le i < j \le n-1} (1+X_i+X_i X_j)(X_j-X_i).
\end{equation}
\end{theo} 

In a sense, \eqref{gfun} is the generating function of order $n$ ASTs $T$, where the exponent of $X_i$ is the position of the $i$-th $1$-column (disregarding the central column) and the position is computed as described in the theorem. However, the coefficient of $X_1^{j_1} X_2^{j_2} \ldots X_{n-1}^{j_{n-1}}$ for non-increasing sequences $j_1,j_2,\ldots,j_{n-1}$ is not necessarily $0$ although there exists no associated AST.

Alternating sign matrices are in bijective correspondence with \emph{fully packed loop configurations} \cite{ProppFaces} (see also Figure~\ref{FPL}), and this relation provides one (of several) possibilities to assign a \emph{``Catalan object''} (in this case a \emph{link pattern}) to an ASM. In  \cite{AignerFPSAC17}, it was recently pointed out that the numbers $\ast(n;j_1,\ldots,j_{n-1})$ are only non-zero if the positions $j_1,\ldots,j_{n-1}$ lie in a certain set whose cardinality is given by the $n$-th Catalan number. Indeed, the natural combinatorial range for the parameters $j_1,j_2,\ldots,j_{n-1}$ is of course 
$$
\{(j_1,j_2,\ldots,j_{n-1}) | 0 \le j_1 < j_2 < \ldots < j_{n-1} \le 2n-3 \},
$$
which is a set of cardinality $\binom{2n-2}{n-1}$, however, $\ast(n;j_1,\ldots,j_{n-1})$ vanishes outside of the following subset: $\ast(n;j_1,\ldots,j_{n-1})$ is only non-zero if, for all $i \in \{1,2,\ldots,n-1\}$, we have  
\begin{equation}
\label{catalan}
|[n-1-i,n-2+i] \cap \{j_1,j_2,\ldots,j_{n-1}\}| \ge i,
\end{equation}
where $[a,b]=\{a,a+1,\ldots,b\}$.
In order to see this, fix an AST $T$ of order $n$ and, for each $i \in \{1,2,\ldots,n-1\}$, consider the triangular array $T_i$ that consists of the $i+1$ bottom rows of $T$. The sum of entries in $T_i$ is $i+1$ as each row sum is $1$. As each column in $T_i$ sums to $1$, $-1$ or $0$, there must be at least $i+1$ columns in $T_i$ that sum to $1$. Since each such column is also a $1$-column in $T$ and the central column of $T$ is always a $1$-column, \eqref{catalan} follows.

One of the ``intermediate'' formulas for $\ast(n;j_1,\ldots,j_{n-1})$ (see Theorem~\ref{ASTcor}) on our way to prove Theorem~\ref{main} enables us to deduce the following identity for the numbers $\ast(n;j_1,\ldots,j_{n-1})$ involving a cyclic rotation  of the parameters $j_1,j_2,\ldots,j_{n-1}$. 

\begin{theo}
\label{identity}
Suppose $n$ is a positive integer and $0 \le j_1 < j_2 < \ldots j_{n-1} \le 2n-3$ such that $j_{n-1} < 2n-3-j_{1}$. Then 
$$
\ast(n;j_1,\ldots,j_{n-1})=
 \sum_{l=2n-3-j_1}^{2n-3} (-1)^{l+1} \binom{l+1}{2n-j_1-2} \ast(n;j_2,\ldots,j_{n-1},l).
$$
\end{theo}

Since the proof of the theorem has again no combinatorial flavor (which is typical for results related to ASMs), it would certainly be of interest to find a combinatorial proof. In view of the relation of Theorem~\ref{main} to a formula for the number of fully packed loop configurations with prescribed link pattern as indicated in Section~\ref{FPLsec}, this identity might be related to the 
fact that the fully packed loop configuration numbers are invariant under the rotation of the link pattern \cite{wieland}.\footnote{By the way, this is one of the few non-trivial results related to ASMs that has a combinatorial proof.} 

Most of the paper is devoted to the proof of Theorem~\ref{main} and to then deduce Theorem~\ref{rogerrefined} from it. In addition, we also indicate how these techniques can possibly be applied to obtain similar results for other classes of ``extreme'' diagonally and antidiagonally alternating sign matrices of odd order as considered in \cite{EXTREMEDASASM}. 
Moreover, we review a formula that was discovered by Di Francesco and Zinn-Justin for the number of fully packed loop configurations associated with a fixed link pattern as these formulas are in terms of linear combinations of coefficients 
of monomials $X_1^{j_1} X_2^{j_2} \ldots X_{n-1}^{j_{n-1}}$ in \eqref{gfun} where the sequences $j_1,j_2,\ldots,j_{n-1}$ are now strictly \emph{decreasing}.

\section{Preliminaries: truncated monotone triangles}

Recall that a \emph{Gelfand-Tsetlin pattern} is a triangular array 
$(m_{i,j})_{1 \le j \le i \le n}$ of integers, where the elements are usually arranged as follows 
\begin{equation}
\label{indexing}
\begin{array}{ccccccccccccccccc}
  &   &   &   &   &   &   &   & m_{1,1} &   &   &   &   &   &   &   & \\
  &   &   &   &   &   &   & m_{2,1} &   & m_{2,2} &   &   &   &   &   &   & \\
  &   &   &   &   &   & \dots &   & \dots &   & \dots &   &   &   &   &   & \\
  &   &   &   &   & m_{n-2,1} &   & \dots &   & \dots &   & m_{n-2,n-2} &   &   &   &   & \\
  &   &   &   & m_{n-1,1} &   & m_{n-1,2} &  &   \dots &   & \dots   &  & m_{n-1,n-1}  &   &   &   & \\
  &   &   & m_{n,1} &   & m_{n,2} &   & m_{n,3} &   & \dots &   & \dots &   & m_{n,n} &   &   &
\end{array}
\end{equation}
such that there is a weak increase in northeast and southeast direction, i.e., 
$m_{i+1,j} \le m_{i,j} \le m_{i+1,j+1}$ for all $i,j$ with $1 \le j \le i < n$. 
A Gelfand-Tsetlin pattern in which each row is strictly increasing except for possibly the bottom row is said to be a \emph{monotone triangle}. (This definition deviates from the standard definition where also the bottom row needs to be strictly increasing.)

We define certain partial monotone triangles which we call \emph{$(\mathbf{s},\mathbf{t})$-trees}: Let $l,r,n$ be non-negative integers with $l+r \le n$. Suppose 
$\mathbf{s}=(s_1,s_2,\ldots,s_l)$, $\mathbf{t}=(t_{n-r+1},t_{n-r+2},\ldots,t_n)$ are sequences of non-negative integers, where $\mathbf{s}$ is weakly decreasing, while $\mathbf{t}$ is weakly increasing. The shape of an $(\mathbf{s},\mathbf{t})$-tree of order $n$ is obtained from the shape of a monotone triangle with $n$ rows when deleting the bottom $s_i$ entries from the $i$-th NE-diagonal for $1 \le i \le l$  (NE-diagonals are counted from the left) and the bottom $t_i$ entries from the $i$-th SE-diagonal for $n-r+1 \le i \le n$ (SE-diagonals are also counted from the left), see Figure~\ref{st_tree}. We assume in the following that 
there is no interference between the deletion of the entries in the $l$  
leftmost NE-diagonals (as prescribed by $\mathbf{s}$) with the deletion of the entries from the $r$ rightmost SE-diagonals (as prescribed by $\mathbf{t}$).

In an $(\mathbf{s},\mathbf{t})$-tree, an entry $m_{i,j}$ is said to be \emph{regular} if it has a SW neighbour $m_{i+1,j}$ and a SE neighbour $m_{i+1,j+1}$.
We require the following monotonicity properties in an 
$(\mathbf{s},\mathbf{t})$-tree: 
\begin{enumerate}
\item Each regular entry $m_{i,j}$ has to fulfill $m_{i+1,j} \le m_{i,j} \le m_{i+1,j+1}$.
\item Two adjacent regular entries $m_{i,j}, m_{i,j+1}$ in the same row have to be distinct.   
\end{enumerate}
This extends the notion of monotone triangles, as a monotone triangle of order $n$ is just an $(\mathbf{s},\mathbf{t})$-tree, where $l,r$ are any two numbers 
with $l+r \le n$, $\mathbf{s}=(\underbrace{0,\ldots,0}_{l})$ and $\mathbf{t}=(\underbrace{0,\ldots,0}_r)$.

We fix some notation that is needed in the following: We use the \emph{shift operator} $\e_x$, the \emph{forward difference} $\fd_{x}$ and the \emph{backward difference} $\bd_{x}$, which are defined as follows.
\begin{align*}
\e_x p(x) &:= p(x+1) \\
\fd_{x} &:= \e_x - \id \\
\bd_{x} &:= \id - \e_x^{-1}
\end{align*} 
The following polynomial in $\mathbf{x}=(x_1,\ldots,x_n)$ plays an important role throughout the paper.
$$
M_n(\mathbf{x}):=\prod_{1 \le p < q \le n} \left (1 + \fd_{x_q} + \fd_{x_p} \fd_{x_q} \right)
\prod_{1 \le i < j \le n} \frac{x_j-x_i}{j-i}
$$

\begin{figure}
\scalebox{0.3}{
\includegraphics{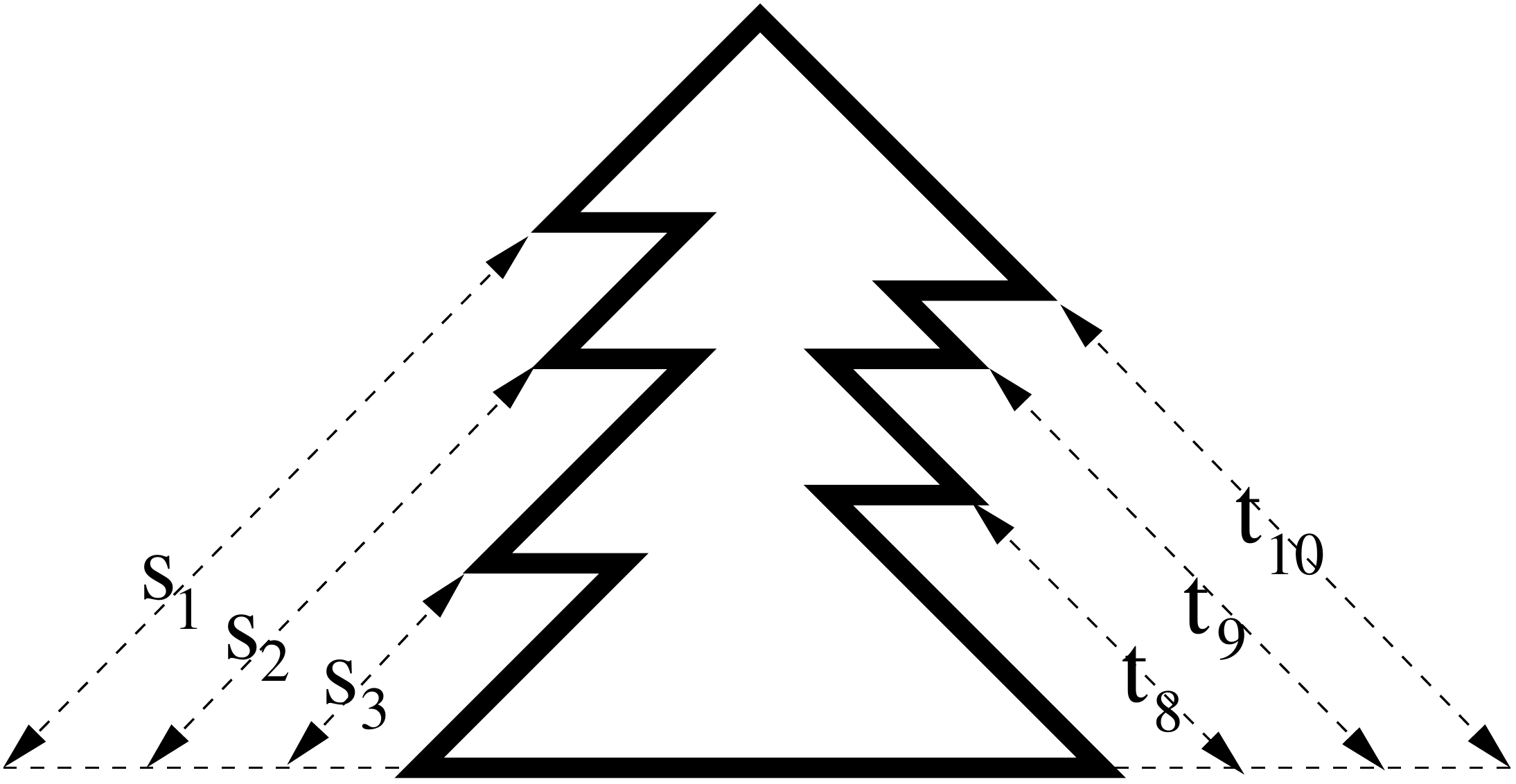}}
\caption{\label{st_tree} Shape of an $(\mathbf{s},\mathbf{t})$-tree.}
\end{figure}

The following theorem is the main result of \cite{FischerRefEnumASM} and is based on previous work in \cite{FischerNumberOfMT}.

\begin{theo} 
\label{truncated}
Let $n,l,r$ be non-negative integers with $l+r \le n$. Suppose $b_1,\ldots,b_n$ is a weakly increasing sequence of integers, and $\mathbf{s}=(s_1,s_2,\ldots,s_l)$, $\mathbf{t}=(t_{n-r+1},t_{n-r+2},\ldots,t_n)$ are a
weakly decreasing and a weakly increasing sequence of non-negative integers, respectively.
Then the evaluation of the following polynomial
\begin{equation}
\label{stoperator}
(-\fd_{x_1})^{s_1} \cdots (-\fd_{x_l})^{s_l} 
\bd_{x_{n-r+1}}^{t_{n-r+1}} 
  \cdots \bd_{x_{n}}^{t_{n}}  
 M_n(\mathbf{x})
\end{equation}
at $(x_1,\ldots,x_n)=(b_1,\ldots,b_n)$ 
is the number of $(\mathbf{s},\mathbf{t})$-trees of order $n$ with the following properties:
\begin{itemize}
\item For $1 \le i \le n-r$, the  bottom entry of the $i$-th NE-diagonal is $b_i$.
\item For $n-r+1 \le i \le n$, the bottom entry of the $i$-th SE-diagonal is $b_i$.
\end{itemize}
\end{theo}

\section{Application of Theorem~\ref{truncated} to ASTs}

We apply Theorem~\ref{truncated} to derive a first formula for the number of ASTs with prescribed positions of the $1$-columns. 

\begin{theo}
\label{ASTcor}
Let $-n+1 \le j_1 < j_2 < \ldots < j_{m-1} < 0 < j_{m+1} < \ldots < j_n \le n-1$. The generating function of ASTs of order $n$ that have the $1$-columns in positions $j_1,j_2,\ldots,j_{m-1},0,j_{m+1},\ldots,j_n$ with respect to the weight 
\begin{equation}
\label{weight}
P^{\text{$\#$ of $10$-columns left of the center}} Q^{\text{$\#$ of $10$-columns right of the center}} 
\end{equation}
is 
\begin{equation}
\label{formula1}
\left. \prod_{i=1}^{m-1} (1-P \bd_{x_i}) (-\bd_{x_i})^{-j_i-1} \prod_{i=m+1}^{n} (1+Q \fd_{x_i}) \fd_{x_{i}}^{j_{i}-1}  M_{n-1}(x_1,\ldots,\widehat{x_m},\ldots,x_n) \right|_{\mathbf{x}=0},
\end{equation}
where the indexing of the columns of ASTs of order $n$ is from left to right with $-n+1,-n+2,\ldots,n-1$. 
\end{theo} 

\begin{proof} 
We perform the standard procedure to transform ASTs into
``monotone triangle structures'' (that is, we modify the bijection between ASMs and monotone triangles), which will in fact turn out to be 
$(\mathbf{s},\mathbf{t})$-trees (an example can be found at the end of the proof): replace each entry in the triangle by the partial columnsum of the entries that lie above the entry and add also the entry itself to this sum. Now record the columns of the $1$'s row by row from top to bottom: to be more precise, if we have reached row $i$---counted from the \emph{bottom} starting with $1$---and the first entry in the previous row  is $-i$ (which is the case if and only if column $-i$ is a $1$-column) then start row $i$ southeast of this entry, otherwise put the first entry of the new row southwest of the first entry in the previous row. Note that the first entry of 
row $i+1$ is $-i$ precisely if column $-i$ is a $1$-column.

A crucial 
point is the following: if the last entry in row $i+1$ is $i$ (which is precisely the case when $i$ is a $1$-column), then the last entry in row $i$ is situated southwest of the last entry of row $i+1$ and otherwise it is situated southeast. This follows from the fact that 
the row sums of ASTs are $1$.

This establishes a bijection with $(\mathbf{s},\mathbf{t})$-trees of order $n$ where
$\mathbf{s}=(-j_1,-j_2,\ldots,-j_{m-1})$ and $\mathbf{t}=(j_{m+1},\ldots,j_n)$
that have the following properties:
\begin{itemize}
\item For $1 \le i \le m-1$, the bottom entry of the $i$-th NE-diagonal is $j_{i}$.
\item For $m+1 \le i \le n$, the bottom entry of the $i$-th SE-diagonal is $j_i$.
\end{itemize}
Theorem~\ref{truncated} now implies that the number of ASTs of order $n$ with $1$-columns precisely in positions $j_1,\ldots,j_{m-1},0,j_{m+1},\ldots,j_n$ is 
$$
\left. \prod_{i=1}^{m-1} (-\fd_{x_i})^{-j_i} \prod_{i=m+1}^{n}  \bd_{x_{i}}^{j_{i}} M_n(\mathbf{x}) \right|_{(x_1,\ldots,x_n)=(j_1,\ldots,j_{m-1},0,j_{m+1},\ldots,j_n)}.
$$

It is instructive to consider the following example.
$$
\begin{array}{ccccccccccccccc}
 &   & 0 & 0 & 0 & 1 & 0 & 0 & 0 & 0 & 0 & 0 & 0 & &  \\
 &   &  & 1 & 0 & -1 & 0 & 0 & 0 & 1 & 0 & 0 & &  &  \\
 &   &  &  & 0 & 0 & 0 & 1 & 0 & -1 & 1 & &  &  &  \\
 &   &  &  &  & 1 & 0 & -1 & 0 & 1 & &  &  &  &  \\
  &   &  &  &  &  & 1 & 0 & 0 & &  &  &  &  &  \\
  &   &  &  &  &  &  & 1 & &  &  &  &  &  &  \\
  &   &  &  &  &  &  &  &  &  &  &  &  &  &
\end{array}
$$
After computing the partial columnsums, we have 
$$
\begin{array}{ccccccccccccccc}
 &   & 0 & 0 & 0 & 1 & 0 & 0 & 0 & 0 & 0 & 0 & 0 &  &  \\
 &   &  & 1 & 0 & 0 & 0 & 0 & 0 & 1 & 0 & 0 &  &  &  \\
 &   &  &  & 0 & 0 & 0 & 1 & 0 & 0 & 1 &  &  &  &  \\
 &   &  &  &  & 1 & 0 & 0 & 0 & 1 &  &  &  &  &  \\
  &   &  &  &  &  & 1 & 0 & 0 &  &  &  &  &  &  \\
  &   &  &  &  &  &  & 1 &  &  &  &  &  &  &  
\end{array}.
$$
We obtain the following $((4,2,1),(2,3))$-tree.
\begin{center}
\begin{tabular}{ccccccccccccccccc}
  &   &   &   &   &   &   &   & $-2$ &   &   &   &   &   &   &   & \\
  &   &   &   &   &   &   & $-4$ &   & $2$ &   &   &   &   &   &   & \\
  &   &   &   &   &   & $\bullet$ &   & $0$ &   & $3$ &   &   &   &   &   & \\
  &   &   &   &   & $\bullet$ &   & $-2$ &   & $2$ &   & $\bullet$ &   &   &   &   & \\
  &   &   &   & $\bullet$ &   & $\bullet$ &  &   $-1$ &   & $\bullet$   &  & $\bullet$  &   &   &   & \\
  &   &   & $\bullet$ &   & $\bullet$ &   & $\bullet$ &   & $0$ &   & $\bullet$ &   & $\bullet$ &   &   &
\end{tabular}.
\end{center}

Now suppose $L_Z \subseteq \{1,\ldots,m-1\}$ and $R_Z \subseteq \{m+1,\ldots,n\}$, then it can be seen that,
by Theorem~\ref{truncated}, the number of ASTs of order $n$ with $1$-columns in positions $j_1,\ldots,j_{m-1},0,j_{m+1},\ldots,j_n$ such that column $j_i$ is a $10$-column if and only if $i \in L_Z \cup R_Z$ is 
\begin{multline}
\label{first}
\prod_{i \in L_Z} (-\fd_{x_i}) \prod_{1 \le i \le m-1 \atop i \notin L_z} E_{x_i} 
\prod_{i \in R_Z} \bd_{x_i} \prod_{m+1 \le i \le n \atop i \not \in R_z} E_{x_i}^{-1} \\
\left. \prod_{i=1}^{m-1} (-\fd_{x_i})^{-j_i} \prod_{i=m+1}^{n}  \bd_{x_{i}}^{j_{i}} M_n(\mathbf{x}) \right|_{(x_1,\ldots,x_n)=(j_1,\ldots,j_{m-1},0,j_{m+1},\ldots,j_n)}.
\end{multline}
It also follows from Theorem~\ref{truncated} that 
$$
\prod_{i=1}^{m-1} (-\fd_{x_i}) \prod_{i=m+1}^{n} \bd_{x_i} M_n(x_1,\ldots,x_n) 
= M_{n-1}(x_1,\ldots,\widehat{x_m},\ldots,x_n)
$$
and, together with $\fd_x = E_x \bd_x$, we can conclude that \eqref{first} is equal to
$$
\left. \prod_{i \in L_z} (-\bd_{x_i}) \prod_{i \in R_Z} \fd_{x_i}
\prod_{i=1}^{m-1} (-\bd_{x_i})^{-j_i-1} \prod_{i=m+1}^{n} \fd_{x_{i}}^{j_{i}-1} M_{n-1}
(x_1,\ldots,\widehat{x_m},\ldots,x_n)) \right|_{\mathbf{x}=0}.
$$
Therefore, the number of ASTs of order $n$ that have the $1$-columns in the prescribed positions,  
$p$ $10$-columns left of the central column and $q$ $10$-columns right of the central column is equal to 
\begin{multline*}
e_p(-\bd_{x_1},\ldots,-\bd_{x_{m-1}}) e_q(\fd_{x_{m+1}},\ldots,\fd_{x_n})
\prod_{i=1}^{m-1} (-\bd_{x_i})^{-j_i-1} \prod_{i=m+1}^{n} \fd_{x_{i}}^{j_{i}-1} \\\\ \left.  M_{n-1}
(x_1,\ldots,\widehat{x_m},\ldots,x_n)) \right|_{\mathbf{x}=0},
\end{multline*}
where $e_r(X_1,\ldots,X_m)$ denotes that $p$-th elementary symmetric function in $X_1,\ldots,X_m$.
The assertion now follows as $e_r(X_1,\ldots,X_m)$ is the coefficient 
of $R^r$ in $\prod\limits_{i=1}^{m} (1+R X_i)$.
\end{proof}

\section{Application of Theorem~\ref{truncated} to odd OOSASM triangles}

In this section we indicate how the technique that was used in the proof of 
Theorem~\ref{ASTcor} can also be applied to other triangular alternating sign arrays (see for instance \cite{EXTREMEDASASM}) by considering one other type of array. We believe that it can be extended further to deal with trapezoidal alternating sign arrays as considered in  \cite{BehrendFischer,AignerFPSAC17} and this will be the subject of a forthcoming paper. 

We consider the following triangular alternating sign arrays that are in (trivial) bijective correspondence with off-diagonally and off-antidiagonally symmetric ASMs (OOSASMs) of odd order as considered in \cite{EXTREMEDASASM}. OOSASMs of even order have been introduced by Kuperberg \cite{KuperbergRoof} and are defined as diagonally and antidiagonally symmetric ASMs (DASASMs)  of even order that have only $0$'s on the diagonal and antidiagonal. As the central entry of an DASASM of odd order is always non-zero, OOSASMs of odd order are defined as DASASMs of odd order that have only $0$'s on the diagonal and antidiagonal with the exception of the central entry.

\begin{defi}
\label{OOSASM}
An \emph{odd OOSASM} triangle of order $n$ is a triangular array 
$$
\begin{array}{cccccccc}
a_{1,1}  & a_{1,2} & a_{1,3} & \ldots & \ldots &  \ldots & a_{1,2n-1} \\
& a_{2,2} & a_{2,3} &\ldots & \ldots & a_{2,2n-2} & \\
& & \ldots & \ldots & \ldots & & \\
&&& a_{n,n} &&&
\end{array}
$$
with $a_{i,j} \in \{0, 1, -1\}$ such that the following conditions are satisfied.
\begin{enumerate}
\item\label{OOSASM-1} For each $j \in \{0,1,\ldots,n-1\}$, the non-zero elements in the following sequence 
(which is a concatenation of the $j$-th column, the $(j+1)$-st row and the $(2n-j)$-th column)
$$
a_{1,j},a_{2,j},\ldots,a_{j,j},a_{j+1,j+1},a_{j+1,j+2},\ldots,a_{j+1,2n-j-1},a_{j,2n-j},a_{j-1,2n-j},\ldots,a_{1,2n-j}
$$
alternate and add up to $1$.
\item\label{OOSASM-2} In the central column, the non-zero elements alternate and the top non-zero element is 
$1$ (if such an element exists at all).
\end{enumerate}
\end{defi}
Suppose $A=(a_{i,j})_{1 \le i,j \le 2n+1}$ is an OOSASM, then the restriction $(a_{i,j})_{1 \le i \le n, i+1 \le j \le 2n+1-i}$ is an odd OOSASM triangle of order $n$ and this mapping is obviously a bijection. 

The sum of entries in the central column of an odd OOSASM triangle of order $n$ is $0$ if $n$ is even and $1$ otherwise. In order to see this note that every entry of an odd OOSASM triangle $T$ of order $n$ is contained in  precisely two sequences of the type as given in Definition~\ref{OOSASM}~(\ref{OOSASM-1}), except for those that are in the central column which are only included in one such sequence. It follows  
that twice the sum of entries in $T$ is $n+c$, where 
$c$ is the sum of entries in the central column. It also follows that there are always $\lfloor n/2 \rfloor$ $1$-columns not including the central column. We can derive a formula for the number of odd OOSASMs triangles as follows.
\begin{theo}
\begin{enumerate}
\item  Let $-2N+1 \le j_1 < j_2 < \ldots < j_{m} < 0 < j_{m+1} < \ldots < j_N \le 2N-1$. The number of odd
OOSASM triangles $T$ of order $2N$ that have the 
$1$-columns in positions $j_1,j_2,\ldots,j_N$ (where the columns are indexed from left to right with $-2N+1,-2N+2,\ldots, 2N-1$) is
\begin{multline*}
(-\bd_{x_1})^{-j_1} (-\bd_{x_2})^{-j_1-1} \ldots (-\bd_{x_{2m-1}})^{-j_{m}}  (-\bd_{x_{2m}})^{-j_{m}-1} \fd_{x_{2m+1}}^{j_{m+1}-1}  \fd_{x_{2m+2}}^{j_{m+1}} \ldots \fd_{x_{2 N -1}}^{j_N-1}  \fd_{x_{2 N}}^{j_N}\\  
\left. M_{2N}(x_1,\ldots,x_{2 N}) \right|_{\mathbf{x}=0}.
\end{multline*}
\item Let $-2N \le j_1 < j_2 < \ldots < j_{m} < 0 < j_{m+1} < \ldots < j_N \le 2N$. The number of odd
OOSASM triangles $T$ of size $2N+1$ that have $1$-columns in positions $j_1,j_2,\ldots,j_{m},0,j_{m+1},\ldots,j_N$ (again columns are indexed from left to right with $-2N, -2N+1,\ldots, 2N$) is
\begin{multline*}
(-\bd_{x_1})^{-j_1} (-\bd_{x_2})^{-j_1-1} \ldots (-\bd_{x_{2m-1}})^{-j_{m}}  (-\bd_{x_{2m}})^{-j_{m}-1} \fd_{x_{2m+2}}^{j_{m+1}-1}  \fd_{x_{2m+3}}^{j_{m+1}} \ldots \fd_{x_{2 N}}^{j_N-1}  \fd_{x_{2 N+1}}^{j_N} \\
\left. M_{2N+1}(x_1,\ldots,x_{2N+1}) \right|_{\mathbf{x}=0}.
\end{multline*}
\end{enumerate}
\end{theo}
\begin{proof}
We consider the $01$-array of partial column sums as in the proof of Theorem~\ref{ASTcor}. Row by row, we record the positions of the $1$'s, from top to bottom. Suppose we are in row $i$, counted from the bottom: if 
the sum of entries in column $-i-1$ is $1$ in the odd OOSASM triangle then we start southeast of the second element of the previous row, otherwise we start southwest of the first element. 

If $n$ is even this is a bijection between $(\mathbf{s},\mathbf{t})$-trees of order $n$ with 
$\mathbf{s}=(-j_1,-j_1-1,-j_2,-j_2-1,\ldots,-j_{m}-1,-j_m)$ and 
$\mathbf{t}=(j_{m+1}-1,j_{m+1},\ldots,j_n-1,j_n)$
with the following properties:
\begin{itemize}
\item The bottom entry is $0$.
\item For $1 \le i \le m$, the bottom entry of the $(2i-1)$-th NE-diagonal is $j_{i}$.
\item For $1 \le i \le m$, the bottom entry of the $2i$-th NE-diagonal is 
$j_{i}+1$.
\item For $m+1 \le i \le N$, the bottom entry of the $(2i-1)$-th SE-diagonal is $j_i-1$.
\item For $m+1 \le i \le N$, the bottom entry of the $2i$-th SE-diagonal is $j_i$.
\end{itemize}
Note that the bottom entries of the even NE-diagonals and the odd SE-diagonals 
in the construction have to be added. By Theorem~\ref{truncated}, the number of these objects is 
\begin{multline*}
(-\fd_{x_1})^{-j_1} (-\fd_{x_2})^{-j_1-1} \ldots (-\fd_{x_{2m-1}})^{-j_{m}} 
(-\fd_{x_{2m}})^{-j_m-1} \bd_{x_{2m+1}}^{j_{m+1}-1} \bd^{j_{m+1}}_{x_{2m+2}}\ldots 
\bd_{x_{2N-1}}^{j_N-1} \bd_{x_{2N}}^{j_N} \\
\left. 
M_{2N}(x_1,\ldots,x_{2N}) \right|_{\mathbf{x}=(j_1,j_1+1,j_2,j_2+1,\ldots,j_N-1,j_N)}.
\end{multline*}
The case $n$ is odd is left to the reader.

We consider the following example.
$$
\small
\begin{array}{ccccccccccccccccc}
&  & 0  & 0 & 1 & 0 & 0 & 0 & 0 & 0 & 0 & 0 & 0 & 0 & 0 &   & \\
& &   & 0 & 0 & 0 & 0 & 0 & 0 & 0 & 1 & 0 & 0 & 0 &  &  & \\
& &   &  & 0 & 0 & 1 & 0 & 0 & 0 & 0 & 0 & 0 &  &  &  & \\
& &   &  &  & 0  & -1 & 0 & 1 & 0 & 0 & 0 &  &  &  &  & \\
& &   &  &  &  & 1 & 0 & 0 & 0 & 0 &  &  &  &  &  & \\
&  & &  & & &  & 0 & -1 & 0 &  &  &  &  &  &  & \\
&  &   &  &  &  &  &  & 1 &  &  &  &  &  &  &  & \\
&  &   &  &  &  &  &  &  &  &  &  &  &  &  & &
\end{array}
$$
After computing the column sums, we obtain
$$
\small
\begin{array}{ccccccccccccccccc}
& & 0  & 0 & 1 & 0 & 0 & 0 & 0 & 0 & 0 & 0 & 0 & 0 & 0 &   & \\
& &   & 0 & 1 & 0 & 0 & 0 & 0 & 0 & 1 & 0 & 0 & 0 & &  & \\
& &   &  & 1 & 0 & 1 & 0 & 0 & 0 & 1 & 0 & 0 &  &  &  & \\
& &   &  &  & 0 & 0 & 0 & 1 & 0 & 1 & 0 & &  &  &  & \\
& &   &  &  &  & 1 & 0 & 1 & 0 & 1 &  &  &  &  &  & \\
&  & &  & & &  & 0 & 0 & 0 &  &  &  &  &  &  & \\
&  &   &  &  &  &  &  & 1 &  &  &  &  &  &  &  & \\
\end{array}.
$$
This is equivalent to the following $((4,3,2,1),(1,2))$-tree.
\begin{center}
\small
\begin{tabular}{ccccccccccccccccc}
  &   &   &   &   &   &   &   & $-4$ &   &   &   &   &   &   &   & \\
  &   &   &   &   &   &   & $-4$ &   & $2$ &   &   &   &   &   &   & \\
  &   &   &   &   &   & $-4$ &   & $-2$ &   & $2$ &   &   &   &   &   & \\
  &   &   &   &   & $\bullet$ &   & $(-3)$ &   & $0$ &   & $2$ &   &   &   &   & \\
  &   &   &   & $\bullet$ &   & $\bullet$ &  &   $-2$ &   & $0$   &  & $2$  &   &   &   & \\
  &   &   & $\bullet$ &   & $\bullet$ &   & $\bullet$ &   & $(-1)$ &   & $(1)$ &   & $\bullet$ &   &   & \\
  &   & $\bullet$ &   & $\bullet$ &   & $\bullet$ &   & $\bullet$ &   & $0$ &   & $\bullet$ &   &  $\bullet$  & \\
\end{tabular}
\end{center}
\end{proof}

\section{Proofs of Theorems~\ref{main} and \ref{identity}}

We return to ASTs and now perform several non-trivial manipulations to the expression in Theorem~\ref{ASTcor} to deduce the following theorem. Most important is the following identity for $M_n(x_1,\ldots,x_n)$ that was first proven in \cite{FischerNewRefProof}.
\begin{equation}
\label{cycle}
M_n(x_1,\ldots,x_n) = (-1)^{n-1} M_n(x_2,\ldots,x_n,x_1-n)
\end{equation}

\begin{theo} 
\label{constantterm1}
The number of ASTs $T$ of order $n$ that have the $1$-columns in positions 
$0 \le j_1 < j_2 < \ldots < j_{n-1} \le 2n-3$, where we exclude the central column and count from the left starting with $0$, such that $\rho(T)=r$ is the coefficient 
of $X_1^{j_1} X_2^{j_2} \ldots X_{n-1}^{j_{n-1}} t^{r-1}$ in
\begin{equation}
\label{formula2}
\prod_{i=1}^{n-1} (t+X_i) \prod_{1 \le i < j \le n-1} (1+X_i+X_i X_j)(X_j-X_i).
\end{equation}
\end{theo}

Theorem~\ref{main} is the special case $t=1$.

\begin{proof}
We apply \eqref{cycle} $m-1$ times to see that the expression in Corollary~\ref{ASTcor} is equal to 
\begin{multline}
\label{second}
(-1)^{(m-1)n}\prod_{i=1}^{m-1} (1-P \bd_{x_i}) (-\bd_{x_i})^{-j_i-1} \e_{x_i}^{-n+1} 
\prod_{i=m+1}^{n} (1+Q \fd_{x_i}) \fd_{x_{i}}^{j_{i}-1}  \\ 
M_{n-1}(x_{m+1},\ldots,x_n,x_1,\ldots,x_{m-1}).
\end{multline}
We use 
$$
1 + \fd_y + \fd_x \fd_y = \e_x \e_y (1 - \bd_x + \bd_x \bd_y) = \e_y 
(1+ \fd_x \bd_y)
$$
to see that the operator in the definition of $M_{n-1}(x_{m+1},\ldots,x_n,x_1,\ldots,x_{m-1})$, that is 
$$
\prod_{1 \le p<q \le m-1} (1+\fd_{x_q}+\fd_{x_p} \fd_{x_q}) 
\prod_{m+1 \le p<q \le n} (1+\fd_{x_q}+\fd_{x_p} \fd_{x_q})
\prod_{p=1}^{m-1} \prod_{q=m+1}^{n} (1+\fd_{x_p}+\fd_{x_p} \fd_{x_q}),
$$
can also be written as
$$
\prod_{p=1}^{m-1} \e_{x_p}^{n-2} 
\prod_{1 \le p<q \le m-1} (1-\bd_{x_p}+\bd_{x_p} \bd_{x_q}) 
\prod_{m+1 \le p<q \le n} (1+\fd_{x_q}+\fd_{x_p} \fd_{x_q})
\prod_{p=1}^{m-1} \prod_{q=m+1}^{n} (1+\bd_{x_p} \fd_{x_q}).
$$ 
Setting
\begin{multline*}
\op(x_1,\ldots,x_n)= \prod_{i=1}^{m-1} (1-P \bd_{x_i}) 
\prod_{i=m+1}^{n} (1+Q \fd_{x_i}) \\ 
\prod_{1 \le p<q \le m-1} (1-\bd_{x_p}+\bd_{x_p} \bd_{x_q}) 
\prod_{m+1 \le p<q \le n} (1+\fd_{x_q}+\fd_{x_p} \fd_{x_q})
\prod_{p=1}^{m-1} \prod_{q=m+1}^{n} (1+\bd_{x_p} \fd_{x_q}),
\end{multline*}
\eqref{second} can also be written as
$$
\op(x_1,\ldots,x_n) \prod_{p=1}^{m-1} \e_{x_p}^{-1}
\prod_{i=1}^{m-1} (-\bd_{x_i})^{-j_i-1}  
\prod_{i=m+1}^{n} \fd_{x_{i}}^{j_{i}-1} \prod_{1 \le i < j \le n \atop i,j \not= m} \frac{x_j-x_i}{j-i}.
$$
As 
$$
\prod_{1 \le i < j \le n \atop i,j \not= m} \frac{x_j-x_i}{j-i} =
\sum_{\sigma \in {\mathcal S}_{n-1}} \sgn \sigma \prod_{i=1}^{m-1} \binom{x_i}{\sigma(i)-1} \prod_{i=m+1}^{n} \binom{x_i}{\sigma(i-1)-1}
$$
it follows that
\begin{multline*}
 \prod_{p=1}^{m-1} \e_{x_p}^{-1}
\prod_{i=1}^{m-1} (-\bd_{x_i})^{-j_i-1}  
\prod_{i=m+1}^{n} \fd_{x_{i}}^{j_{i}-1} \prod_{1 \le i < j \le n \atop i,j \not= m} \frac{x_j-x_i}{j-i} \\
= (-1)^{j_1+\ldots+j_{m-1}+m-1} \sum_{\sigma \in {\mathcal S}_{n-1}} \sgn \sigma 
\prod_{i=1}^{m-1} \binom{x_i+j_i}{\sigma(i)+j_i} \prod_{i=m+1}^{n} \binom{x_i}{\sigma(i-1) -j_i}.
\end{multline*}
Since 
\begin{align*}
(-\bd_{x_i}) \binom{x_i+j_i}{\sigma(i)+j_i}(-1)^{j_i-1} &= \e_{j_i}^{-1} \binom{x_i+j_i}{\sigma(i)+j_i}(-1)^{j_i-1}, \\
\fd_{x_i} \binom{x_i}{\sigma(i-1)-j_i}&=\e_{j_i} \binom{x_i}{\sigma(i-1)-j_i},
\end{align*}
we may replace each $-\bd_{x_i}$ in $\op$ by $\e^{-1}_{j_i}$ and each 
$\fd_{x_i}$ by $\e_{j_i}$. Moreover, we can now evaluate at $(x_1,\ldots,x_n)=0$.
It follows that \eqref{second} is equal to 
\begin{multline}
\label{long}
\prod_{i=1}^{m-1} (1+P \e^{-1}_{j_i}) 
\prod_{i=m+1}^{n} (1+Q \e_{j_i}) \\ 
\prod_{1 \le p<q \le m-1} (1+ \e^{-1}_{j_p}+\e^{-1}_{j_p} \e^{-1}_{j_q}) 
\prod_{m+1 \le p<q \le n} (1+\e_{j_q}+\e_{j_p} \e_{j_q})
\prod_{p=1}^{m-1} \prod_{q=m+1}^{n} (1-\e^{-1}_{j_p} \e_{j_q}) \\
(-1)^{j_1+\ldots+j_{m-1}+m-1} \sum_{\sigma \in {\mathcal S}_{n-1}} \sgn \sigma 
\prod_{i=1}^{m-1} \binom{j_i}{\sigma(i)+j_i} \prod_{i=m+1}^{n} \binom{0}{\sigma(i-1) -j_i}.
\end{multline} 

Now we use the following: Suppose $f(j_1,\ldots,j_k)$ is a function such 
that the generating function 
$$F(X_1,\ldots,X_k):= \sum_{(j_1,\ldots,j_k) \in \mathbb{Z}^k} f(j_1,\ldots,j_k) X_1^{j_1} \ldots X_n^{j_k}$$
is a Laurent series and let $p(X_1,\ldots,X_k)$ be a Laurent polynomial. Then 
\begin{equation}
\label{obs}
p(\e_{j_1},\ldots,\e_{j_k}) f(j_1,\ldots,j_k) = \ct_{X_1,\ldots,X_k} X_1^{-j_1} \ldots X_k^{-j_k} p(X_1^{-1},\ldots,X_k^{-1}) F(X_1,\ldots,X_k),
\end{equation}
where $\ct_{X_1\ldots,X_{k}}$ denotes the constant term of the expression.
We apply this to \eqref{long}  with 
$$
f(j_1,\ldots,j_{m-1},j_{m+1},\ldots,j_n) = (-1)^{j_1+\ldots+j_{m-1}+m-1} \sum_{\sigma \in {\mathcal S}_{n-1}} \sgn \sigma 
\prod_{i=1}^{m-1} \binom{j_i}{\sigma(i)+j_i} \prod_{i=m+1}^{n} \binom{0}{\sigma(i-1) -j_i}.
$$
As 
$$
\sum_{j_i=-n+1}^{\infty} \binom{j_i}{\sigma(i) + j_i} (-1)^{j_i-1} X_i^{j_i} =  X_i^{-1} (-1-X_{i}^{-1})^{\sigma(i)-1}
\quad \text{and} \quad 
\sum_{j_{i+1}=0}^{\infty} \binom{0}{\sigma(i)-j_{i+1}} X_{i+1}^{j_{i+1}} = X_{i+1}^{\sigma(i)},
$$
we can conclude that 
\begin{multline*}
F(X_1,\ldots,\widehat{X_m},\ldots,X_n) = \sum_{(j_1,\ldots,\widehat{j_{m}},\ldots,j_n) \in \mathbb{Z}^{n-1}} f(j_1,\ldots,\widehat{j_{m}},\ldots,j_n) X_1^{j_1} \ldots X_{m-1}^{j_{m-1}} X_{m+1}^{j_{m+1}} \ldots X_n^{j_n} \\ = 
\prod_{i=1}^{m-1} X_i^{-1} \prod_{i=m+1}^n X_i \det_{1 \le i, j \le n-1} \left( 
\begin{cases} (-1-X_i^{-1})^{j-1}, & i=1,\ldots,m-1 \\ X_{i+1}^{j-1}, & i=m,\ldots,n-1 
\end{cases} \right).
\end{multline*}
Using the Vandermonde determinant evaluation, we  obtain 
\begin{multline*}
F(X_1,\ldots,\widehat{X_j},\ldots,X_n) = 
\prod_{i=1}^{m-1} X_i^{-m+1} \prod_{i=m+1}^n X_i \\ 
\times \prod_{1 \le i < j \le m-1} (X_j-X_i) \prod_{i=1}^{m-1} \prod_{j=m+1}^{n} (1+X_i^{-1}+X_j) 
\prod_{m+1 \le i < j \le n} (X_j-X_i).
\end{multline*}
In total, it follows that the generating function of ASTs 
w.r.t.\ $\rho$ that have the $1$-columns in 
columns $-n+1 \le  j_1 < j_2 < \ldots j_{m-1} < 0 < j_{m+1} < \ldots < j_n \le n-1$ is the coefficient of 
$X_1^{j_1} \ldots X_{m-1}^{j_{m-1}} X_{m+1}^{j_{m+1}} \ldots X_{n}^{j_{n}}$ in 
\begin{multline}
\label{before}
 \prod_{i=1}^{m-1} (1+ P X_{i}) \prod_{i=m+1}^{n} (1+Q X_{i}^{-1})  \prod_{i=1}^{m-1} X_i^{-1} \prod_{i=m+1}^{n} X_i 
\\ \times \prod_{1 \le i < j \le n \atop i,j \not= m} (1 +X_{j}^{-1} + X_{i}^{-1} X_{j}^{-1}) (X_j - X_i)
\end{multline}
For fixed AST $T$ of order $n$ with $m-1$ $1$-columns left of the central column, let be $p$ the number of $10$-columns left of the central column and $q$ be the number of $11$-columns right of the central column. Then 
$$
\rho(T) = m-p+q.
$$
This together with \eqref{before} and taking into account the different indexing of the $1$-columns implies the statement in the theorem.
\end{proof}

It follows that the number of ASTs $T$ of order $n$ with $\rho(T)=r$ is the constant term of
$$
t^{-r+1} \sum_{0 \le j_1<j_2< \ldots < j_{n-1}} \prod_{i=1}^{n-1} (t+X_i) X_i^{-j_i} \prod_{1 \le i < j \le n-1} (1+X_i+X_i X_j)(X_j-X_i).
$$
in $X_1,\ldots,X_{n-1},t$, or, equivalently, the constant term of 
$$
t^{-r+1} \sum_{0 \le j_1<j_2< \ldots < j_{n-1}} \prod_{i=1}^{n-1} (t+X_i^{-1}) X_i^{j_i} 
\prod_{1 \le i < j \le n-1} (1+X_i^{-1}+X_i^{-1} X_j^{-1})(X_j^{-1}-X_i^{-1}).
$$
This implies the following result, which is the starting point for our proof of Theorem~\ref{rogerrefined}.

\begin{cor} 
\label{refinedCor}
The number of ASTs $T$ of order $n$ with 
$\rho(T)=r$ is equal to the constant term in $X_1,\ldots,X_{n-1},t$ of 
$$
\frac{ t^{-r+1} \prod\limits_{i=1}^{n-1} (1+t X_i) X_i^{i-2n+2}  \prod\limits_{1 \le i < j \le n-1} (1+X_j+X_i X_j)(X_i-X_j)}{\prod\limits_{i=1}^{n-1} 
\left(1- \prod\limits_{j=i}^{n-1} X_j \right)}
$$ 
when interpreting this expression as a formal laurent series with $\left(1- \prod\limits_{j=i}^{n-1} X_j\right)^{-1}= \sum\limits_{k=0}^{\infty} \prod\limits_{j=i}^{n-1} X_j^k$.
\end{cor}

In Theorems~\ref{ASTcor} and \ref{constantterm1}, the parameters $j_i$ index the $1$-columns of ASTs differently. In order to pass from a sequence $j_1,\ldots,\widehat{j_m},\ldots,j_{n}$ in Theorem~\ref{ASTcor} to the equivalent sequence in Theorem~\ref{constantterm1}, one has to add $n-1$ to all $j_i$ that are negative and add $n-2$ to all $j_i$ that are positive. More formally, we have 
$$
\tast(n,m;j_1,\ldots,\widehat{j_m},\ldots,j_n) = \ast(n;j_1+n-1,\ldots,j_{m-1}+n-1,j_{m+1}+n-2,\ldots,j_n+n-2),
$$
where $\tast(n,m;j_1,\ldots,\widehat{j_m},\ldots,j_n)$ denotes the quantity in the statement of Theorem~\ref{ASTcor}.

It is useful to note that the equivalence of the formulas \eqref{formula1} and \eqref{formula2}  was not only shown for (combinatorially meaningful) sequences as indicated in the theorem, but for the following more general sequences: In Theorem~\ref{ASTcor}, we may extend to sequences of integers that do not include $0$ and where all negative values appear before the positive values, while in Theorem~\ref{constantterm1} we can extend to sequences such that all values smaller than $n-1$ appear before all values greater than or equal to $n-1$. We are now in the position to prove Theorem~\ref{identity}, which is in fact true on the extended set of sequences of integers $j_1,\ldots,j_{n-1}$ that we have just described. In fact, computer experiments suggest that it is true for all integers sequences when  $\ast(n;j_1,\ldots,j_{n-1})$ is defined to be 
the coefficient of $X_1^{j_1} \ldots X_{n-1}^{j_{n-1}}$ in \eqref{gfun}.

\begin{proof}[Proof of Theorem~\ref{identity}] Suppose $m \in \{1,2,\ldots,n\}$ and $j_1,\ldots,\widehat{j_m},\ldots,j_n$ is a sequence such that $j_i$ is negative for $i<m$ and positive for $i>m$. 
Assume $m \not=1$. Then, by \eqref{cycle},
\begin{align*}
&\tast(n,m;j_1,\ldots,\widehat{j_m},\ldots,j_n) \\
&= \left. (-1)^{n+j_1+1} \e_{x_1}^{-n+j_1+1} \fd_{x_1}^{-j_1-1} \prod_{i=2}^{m-1} \e_{x_i}^{-1}  (-\bd_{x_i})^{-j_i-1} \prod_{i=m+1}^{n} \e_{x_i}  \fd_{x_{i}}^{j_{i}-1}  M_{n-1}(x_2,\ldots,\widehat{x_m},\ldots,x_n,x_1) \right|_{x=0} \\
&= \left. (-1)^{n+j_1+1} (1+\fd_{x_1})^{-n+j_1} \e_{x_1} \fd_{x_1}^{-j_1-1}
\prod_{i=2}^{m-1} \e_{x_i}^{-1}  (-\bd_{x_i})^{-j_i-1} \prod_{i=m+1}^{n} \e_{x_i}  \fd_{x_{i}}^{j_{i}-1}  M_{n-1}(x_2,\ldots,x_n,x_1) \right|_{x=0} \\
&= \left. (-1)^{n+j_1+1} \e_{x_1} \sum_{k=0}^{\infty} \binom{-n+j_1}{k} \fd_{x_1}^{k-j_1-1}
\prod_{i=2}^{m-1} \e_{x_i}^{-1}  (-\bd_{x_i})^{-j_i-1} \prod_{i=m+1}^{n} \e_{x_i}  \fd_{x_{i}}^{j_{i}-1}  M_{n-1}(x_2,\ldots,x_n,x_1) \right|_{x=0}
\\
&= \left. (-1)^{n+j_1+1} \e_{x_1} \sum_{l=-j_1}^{\infty} \binom{-n+j_1}{l+j_1} \fd_{x_1}^{l-1}
\prod_{i=2}^{m-1} \e_{x_i}^{-1}  (-\bd_{x_i})^{-j_i-1} \prod_{i=m+1}^{n} \e_{x_i}  \fd_{x_{i}}^{j_{i}-1}  M_{n-1}(x_2,\ldots,x_n,x_1) \right|_{x=0}
\end{align*}
We use $\binom{n}{k} = (-1)^{k} \binom{k-n-1}{k}$ and the symmetry of the binomial coefficient to see that this is further equal to 
\begin{align*}
& \left.   \e_{x_1} \sum_{l=-j_1}^{\infty} (-1)^{n+l+1} \binom{l+n-1}{-j_1+n-1} \fd_{x_1}^{l-1}
\prod_{i=2}^{m-1} \e_{x_i}^{-1}  (-\bd_{x_i})^{-j_i-1} \prod_{i=m+1}^{n} \e_{x_i}  \fd_{x_{i}}^{j_{i}-1}  M_{n-1}(x_2,\ldots,\widehat{x_m},\ldots,x_n,x_1) \right|_{x=0} \\
&=\sum_{l=-j_1}^{\infty} (-1)^{n+l+1} \binom{l+n-1}{-j_1+n-1} \tast(n,m-1;j_2,\ldots,\widehat{j_m},\ldots,j_n,l).
\end{align*}
It follows that 
\begin{align*}
&\ast(n;j_1,\ldots,j_{n-1})
\\
&=\sum_{l=-j_1+n-1}^{\infty} (-1)^{n+l+1} \binom{l+n-1}{2n-j_1-2} \\
& \quad \quad \quad \times\tast(n,m-1;j_2-n+1,\ldots,j_{m-1}-n+1,j_{m}-n+2,\ldots,j_{n-1}-n+2,l) \\
&= \sum_{l=2n-3-j_1}^{2n-3} (-1)^{l+1} \binom{l+1}{2n-j_1-2} \ast(n;j_2,\ldots,j_{n-1},l).
\end{align*}
\end{proof}

\section{Proof of Theorem~\ref{rogerrefined}}

We need to define the \emph{symmetrizer} $\sym$ and the \emph{antisymmetrizer} $\asym$ of a function $f(x_1,\ldots,x_n)$ with respect to the variables $x_1,\ldots,x_n$.
\begin{align*}
\sym_{x_1,\ldots,x_n} f(x_1,\ldots,x_n) &= 
\sum_{\sigma \in {\mathcal S}_n} f(x_{\sigma(1)},\ldots,x_{\sigma(n)}) \\
\asym_{x_1,\ldots,x_n} f(x_1,\ldots,x_n) &= 
\sum_{\sigma \in {\mathcal S}_n} \sgn \sigma f(x_{\sigma(1)},\ldots,x_{\sigma(n)})
\end{align*}

\begin{lem} Let $n \ge 1$. Then 
\begin{multline}
\label{asym}
\asym_{x_1,\ldots,x_n} \left[ \prod_{1 \le i < j \le n} (1+X_j+X_i X_j) \prod_{i=1}^{n} X_i^{i-1} \left(1-\prod_{j=i}^{n} X_j\right)^{-1} \right] \\
= \prod_{i=1}^{n} (1-X_i)^{-1} \prod_{1 \le i < j \le n} 
\frac{(1+X_i + X_j)(X_j-X_i)}{(1-X_i X_j)}.
\end{multline}
\end{lem}

\begin{proof} Let $A(X_1,\ldots,X_n)$ denote the left-hand side of \eqref{asym}. We proceed by induction with respect to $n$ and observe that the case $n=1$ is obvious. We have the following recursion for $A(X_1,\ldots,X_n)$.
\begin{multline*}
A(X_1,\ldots,X_n) = \sum_{k=1}^{n} (-1)^{k+1} \left(1-\prod_{l=1}^{n} X_l\right)^{-1} \prod_{1 \le l \le n, l \not= k}  (1+X_l+X_k X_l) X_l \\
\times A(X_1,\ldots,\widehat{X_k},\ldots,X_n)
\end{multline*}
We need to show that the right-hand side of \eqref{asym} satisfies this recursion. A straightforward computation shows that this is equivalent to proving that 
\begin{multline*}
\prod_{1 \le i < j \le n} (1+X_i + X_j) \left(1 - \prod_{i=1}^{n} X_i \right)  \\
= \sum_{k=1}^{n} (1-X_k) \prod_{1 \le i \le n \atop i \not= k} \frac{(1+X_i+X_i X_k)(1-X_i X_k) X_i}{X_i-X_k} \prod_{1 \le i < j \le n \atop i,j \not= k} (1+X_i + X_j).
\end{multline*}
This can be seen by first observing that both sides of the identity are symmetric polynomials in $X_1,\ldots,X_n$ of degree no greater than $n$ in each $X_i$: In order to see this for the right-hand side, the expression is written as a fraction of two polynomials with $\prod_{1 \le i < j \le n} (X_j-X_i)$ in the denominator, and observing that the numerator is antisymmetric and thus divisible by the denominator. Then it can easily be shown that both sides vanish whenever $X_i=-1-X_j$, $i \not = j$ and that their evaluations at $X_i=0$ coincide as wells as the leading coefficients w.r.t.\ lexicographic term order. 
\end{proof}

We now use a trick that I have learned from \cite{ZeilbergerASMProof} where it is attributed to Stanton and Stembridge. From Corollary~\ref{refinedCor}, it follows that the generating function of order $n$ ASTs w.r.t. $\rho$ is the constant term in $X_1,\ldots,X_n$ of the following expression.
\begin{multline*}
\frac{t}{(n-1)!} \sym_{X_1,\ldots,X_{n-1}} \left[ \prod_{i=1}^{n-1} (1+t X_i) X_i^{i-2n+2} \left(1-\prod_{j=i}^{n-1} X_j \right)^{-1} \right. \\ \left. \times \prod_{1 \le i < j \le n-1} (1+X_j+X_i X_j)(X_i-X_j) \right]  \\
= \frac{t}{(n-1)!} \prod_{i=1}^{n-1} (1+t X_i) X_i^{-2n+3} \prod_{1 \le i < j \le n-1} (X_i-X_j) \\ \times \asym_{X_1,\ldots,X_{n-1}} 
\left[ \prod_{1 \le i < j \le n-1} (1+X_j+X_i X_j) \prod_{i=1}^{n-1} X_i^{i-1} \left(1-\prod_{j=i}^{n-1} X_j\right)^{-1} \right]
\end{multline*}
By Lemma~\ref{asym}, this expression is equal to 
\begin{equation}
\label{after-asym}
\frac{t}{(n-1)!}
\prod_{i=1}^{n-1} (1+t X_i) (1-X_i)^{-1} X_i^{-2n+3} \prod_{1 \le i < j \le n-1} 
\frac{(1+X_i + X_j)(X_j-X_i)(X_i-X_j)}{(1-X_i X_j)}.
\end{equation}
Now we use the following identity
$$
\asym_{X_1,\ldots,X_{n-1}} \left[ \prod_{i=1}^{n-1} X_i^{i-1} \left(1-\prod_{j=i}^{n-1} X_j\right)^{-1} \right]
= \prod_{i=1}^{n-1} (1-X_i)^{-1} \prod_{1 \le i < j \le n-1} \frac{X_j-X_i}{1-X_i X_j},
$$
which can be found in  \cite[Subsublemma 1.1.3]{ZeilbergerASMProof} (it can be proven in a similar way as the identity in Lemma~\ref{asym}). It follows that \eqref{after-asym} is equal to 
\begin{multline*}
\frac{t}{(n-1)!} \prod_{i=1}^{n-1} (1+t X_i) X_i^{-2n+3} \prod_{1 \le i < j \le n-1} 
(1+X_i+X_j)(X_i-X_j) \\ \times \asym_{X_1,\ldots,X_{n-1}} \left[ \prod_{i=1}^{n-1} X_i^{i-1} 
\left(1-\prod_{j=i}^{n-1} X_j\right)^{-1} \right] \\ =
\frac{t}{(n-1)!} \sym_{X_1,\ldots,X_{n-1}} \left[ \prod_{i=1}^{n-1} (1+t X_i) X_i^{-2n+2+i} 
\left(1-\prod_{j=i}^{n-1} X_j\right)^{-1} \prod_{1 \le i < j \le n-1} 
(1+X_i+X_j)(X_i-X_j) \right].
\end{multline*}
However, the constant term of the previous expression is equal to the constant term of 
\begin{multline*}
 t \prod_{i=1}^{n-1} (1+t X_i) X_i^{-2n+2+i} 
\left(1-\prod_{j=i}^{n-1} X_j\right)^{-1} \prod_{1 \le i < j \le n-1} 
(1+X_i+X_j)(X_i-X_j) \\
= t \sum_{0 \le b_1 < b_2 < \ldots < b_{n-1}} \prod_{i=1}^{n-1} X_i^{-2n+3+b_i} (1+t X_i) \prod_{1 \le i < j \le n-1}  (1+X_i+X_j)(X_i-X_j).
\end{multline*}
By the Vandermonde determinant evaluation, this is furthermore equal to 
$$
(-1)^{\binom{n-1}{2}} t \sum_{0 \le b_1 < b_2 < \ldots < b_{n-1}}  \det_{1 \le i, j \le n-1}
\left( (1+t X_j) (1+X_j)^{i-1} X_j^{i-2n+2+b_j} \right).
$$
By the binomial theorem, the constant term is 
\begin{multline*}
(-1)^{\binom{n-1}{2}} t \sum_{0 \le b_1 < b_2 < \ldots < b_{n-1}} \det_{1 \le i, j \le n-1} \left( \binom{i-1}{2n-2-i-b_j} + t \binom{i-1}{2n-3-i-b_j} \right) \\
= t \sum_{0 \le b_1 < b_2 < \ldots < b_{n-1}} \det_{1 \le i, j \le n-1} \left( \binom{i-1}{b_j-i+1}  + t \binom{i-1}{b_j-i} \right),
\end{multline*}
where we use the transformation $b_j \to 2n-3-b_{n-j}$.
By the Lindstr\"om-Gessel-Viennot theorem \cite{Lindstr,GesselViennot}, this is the generating function of families of non-intersecting lattice paths with step set $\{(1,0),(0,1)\}$, starting points $(i-1,-2i+1)$, $i=1,2,\ldots,n-1$, and endpoints $(b,-b)$, $b \ge 0$, where the weight of a family is $t$ to the power of the number of lattice paths that start with an east step plus $1$.  An example of such a family of non-intersecting lattice paths is given in Figure~\ref{magog}. 

\begin{figure}
\scalebox{0.3}{
\includegraphics{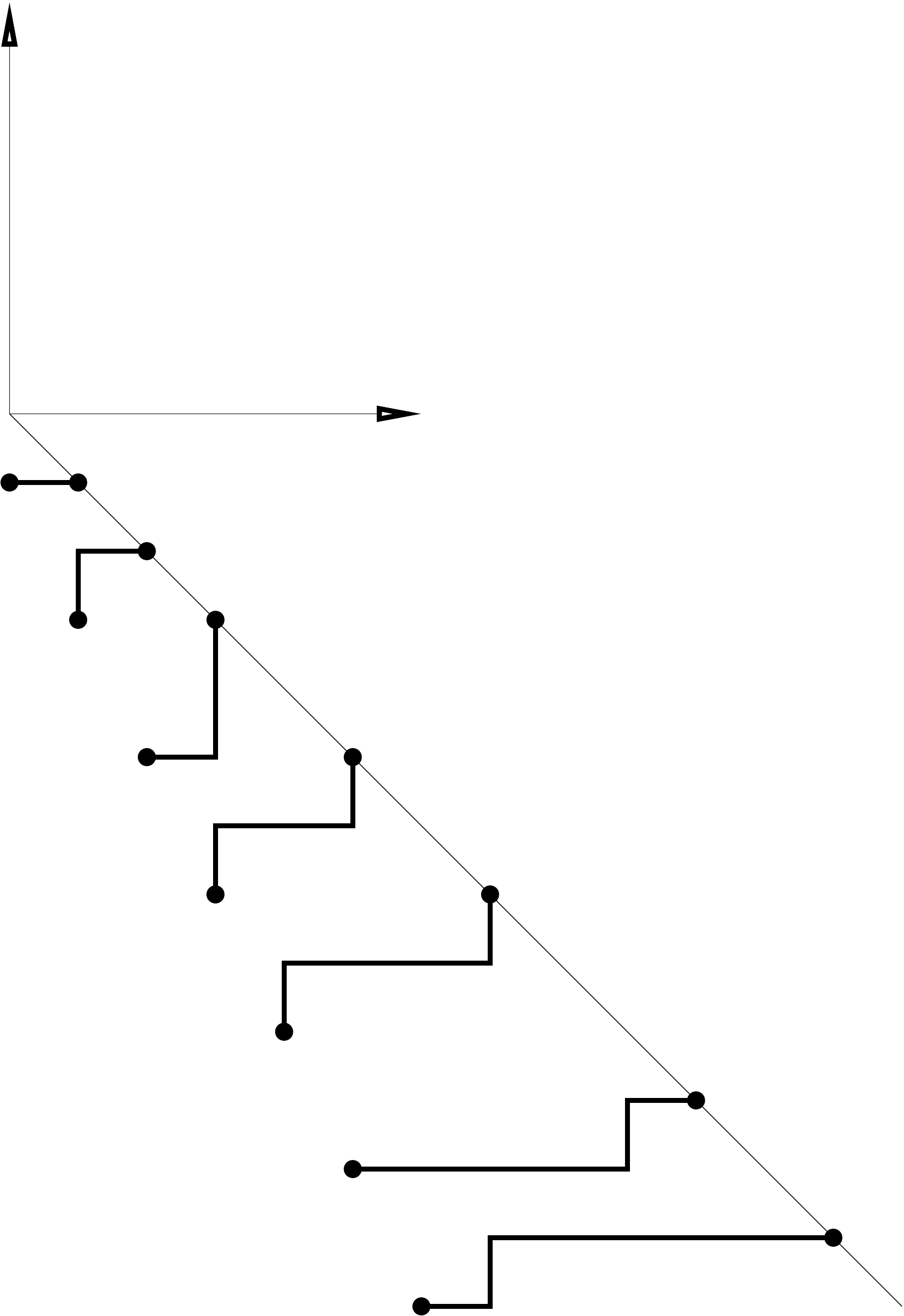}}
\caption{\label{magog} A family of non-intersecting lattice paths for $n=8$.}
\end{figure}

Such families of non-intersecting lattice paths are in bijective correspondence with Gelfand-Tsetlin patterns with $n$ rows and positive integer entries such that the entries in the $i$-th north-east diagonal, counted from the left, are bounded by $i$. The weight translates into $t$ to the power of number of entries in the rightmost south-east diagonal that are equal to the upper bound of their north-east diagonal. An example of such a Gelfand-Tsetlin pattern is given next.
$$
\begin{array}{ccccccccccccccc}
&&&&&&&1&&&&&&& \\
&&&&&&1&&2&&&&&&\\ 
&&&&&1&&1&&3&&&&& \\
&&&&1&&1&&3&&3&&&& \\
&&&1&&1&&3&&3&&4&&& \\
&&1&&1&&3&&3&&4&&6&&\\
&1&&1&&3&&3&&4&&5&&6&\\
1&&1&&2&&3&&3&&4&&5&&8
\end{array}
$$
It corresponds to the family in Figure~\ref{magog}. The $i$-th path in Figure~\ref{magog}, counted from the bottom, corresponds to the separating path (after rotation) between the entries less than or equal to $i$ and the entries greater than $i$ in the corresponding Gelfand-Tsetlin pattern.  It is not difficult to see that such Gelfand-Tsetlin patterns are in bijective correspondence to $2n \times 2n \times 2n$ TSSCPPs \cite{MilRobRum86}. Theorem~\ref{rogerrefined} follows now from the main result in \cite{FonZinn}.

\section{Fully packed loop configurations associated with a fixed link pattern}
\label{FPLsec}

In the final section we point out a relation to a formula for the number of fully packed loop configurations (FPLs) with a prescribed link pattern.
FPLs are defined to be subgraphs of the $n \times n$ square grid with $n$ external edges on each side such that every internal vertex is of degree $2$ and every other external edge is contained in the subgraph. They are in easy bijective correspondence with ASMs \cite{ProppFaces}, see Figure~\ref{FPL} for an example where also the corresponding six-vertex configuration is displayed (as the six-vertex configuration can be seen as an intermediate object when passing from ASMs to FPLs). 

\begin{figure}[ht]
\begin{minipage}{0.33\linewidth}
\centering
\mbox{$\left(\begin{matrix}
0 & 0  & 1 & 0 & 0 \\
1 & 0 & -1 &  1 & 0 \\
0 & 1 & 0 &  -1  & 1 \\
0 & 0 & 1 & 0 & 0 \\
0 & 0 & 0 & 1 & 0
\end{matrix}
\right)$}
\end{minipage}
\begin{minipage}{0.30\linewidth}
\centering
\mbox{\scalebox{0.5}{\includegraphics[width=\textwidth]{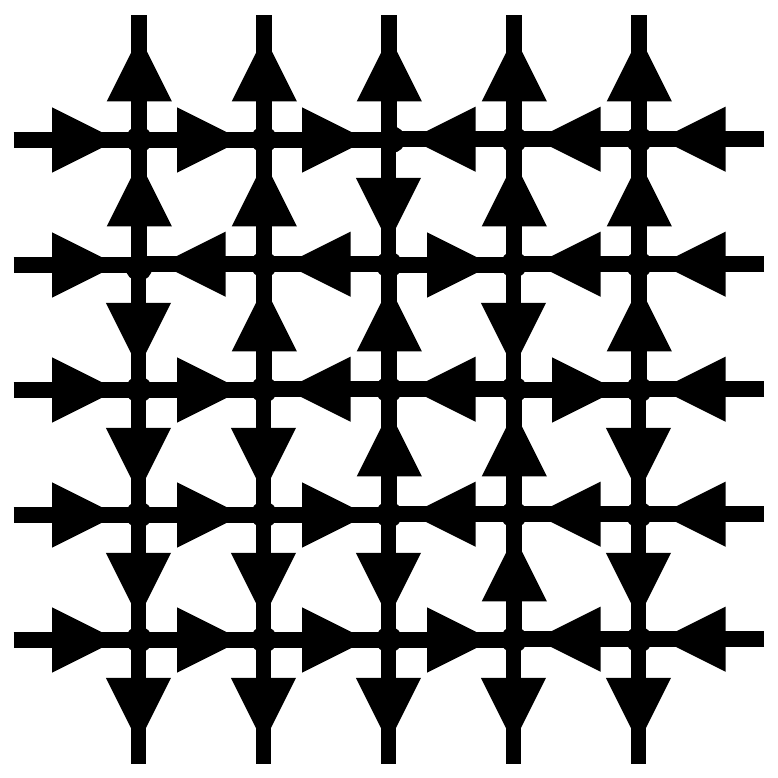}}}
\end{minipage}
\begin{minipage}{0.30\linewidth}
\centering
\mbox{\scalebox{0.5}{\includegraphics[width=\textwidth]{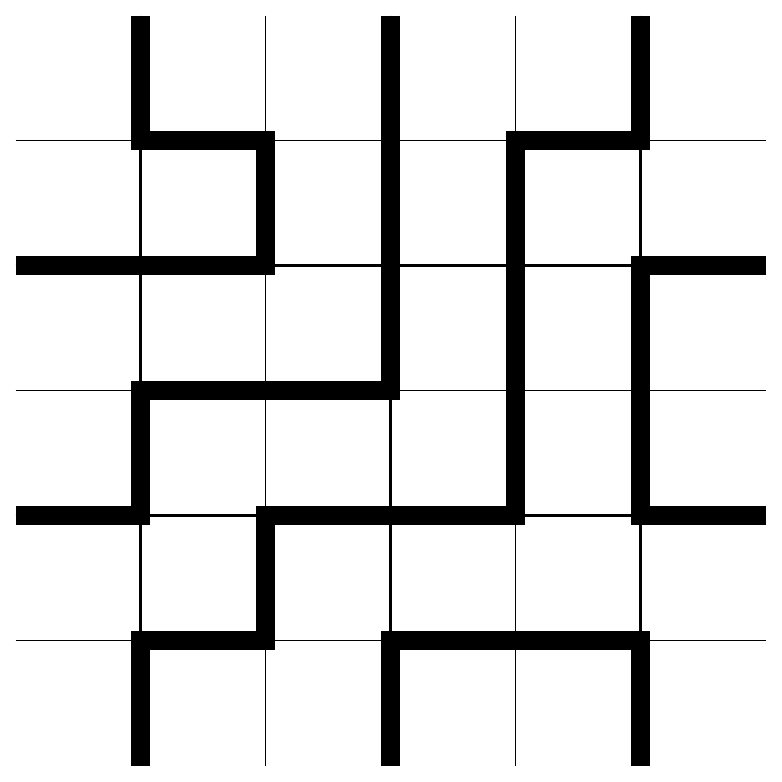}}}
\end{minipage}
\caption{\label{FPL} ASM, $6$-vertex configuration and FPL.}
\end{figure}

The following set is in natural bijection with Dyck words of length $2n$:
$${\mathbf A}_n= \{\alpha=(\alpha_1,\ldots,\alpha_{n-1}) \in \mathbb{Z}^n | 0 \le \alpha_1 < \ldots < \alpha_{n-1} \le 2n-3 \text{ and } \alpha_i \le 2i-1\}$$
Indeed, for a given Dyck word $\mathbf{w}$ of length $2n$, $\alpha_i$ is two less than the position of the $(i+1)$-st $0$ in $\mathbf{w}$ and we write $\alpha(\mathbf{w})$ for the element of $\mathbf{A}_n$ corresponding to $\mathbf{w}$. We let $\Psi_{\alpha}$ denote the coefficient of $X_1^{\alpha_1} \ldots X_{n-1}^{\alpha_{n-1}}$:
\begin{equation}
\label{psi}
 \prod_{i=1}^{n-1} (1+X_i) \prod_{1 \le i < j \le n-1} (1+X_j+X_i X_j)(X_j-X_i)
\end{equation}
This coefficient is the coefficient of $X_1^{2n-3-\alpha_{n-1}} X_2^{2n-3-\alpha_{n-2}}  \ldots X_{n-1}^{2n-3-\alpha_1}$ 
in the generating function in Theorem~\ref{main}. 

Di Francesco and Zinn-Justin \cite{sumrule,QKZ,HabilZinnJustin} (see also 
\cite{Romik2014}) have developed a beautiful theory which implies that the number of fully packed loop configurations associated with a fixed link pattern $\pi$, denoted by $\phi_{\pi}$, can be expressed as a certain linear combination of $\Psi_{\alpha}$'s. 

In order to give the precise statement, we need to define a matrix $C^{\ext}(n)$ whose columns are indexed by link patterns on $\{1,2,\ldots, 2 n\}$ and whose rows are indexed by $01$-words of length $2n$. To fix notation, we let $\mathbf{w}[i,j]$ denote the subword of a given word
$\mathbf{w}$ consisting of the letters in positions $i, i+1,\ldots, j$.
Now suppose $\pi$ is a link pattern and $\mathbf{w}$ is a $01$-word. Consider an arc $\mathbf{a}$ of $\pi$ which connects, say, $i$ and $j$, $i < j$, and let $n_0$ be the number of $0$'s in $\mathbf{w}[i,j-1]$ and $n_1$ be the number of $1$'s in $\mathbf{w}[i,j-1]$. We say that $\mathbf{a}$ is \emph{contradicting} with respect to $\mathbf{w}$ if $n_0-n_1 \equiv -1 \mod 3$, while we say that $\mathbf{a}$ is \emph{negative} if $n_0-n_1 \equiv 0 \mod 3$. The entry corresponding to $\pi,\mathbf{w}$ is now defined as (compare to \cite[Appendix~A]{QKZ} or \cite[Theorem~2.17]{Romik2014})
$$
C^{\ext}(n)_{\mathbf{w},\pi}= \begin{cases} 0, & \text{if $\pi$ contains a contradicting arc}, \\ (-1)^{\# \text{negative arcs}}, & \text{otherwise.} \end{cases}
$$
Note that if $\mathbf{w}$ equals the Dyck word of $\pi$ then we have $n_0 - n_1 = 1$ for all arcs and the entry is $1$ in this case.

Let $C(n)$ denote the square submatrix we obtain by restricting to rows indexed by Dyck words. For $n=3,4$, the matrices are displayed in Figures~\ref{P3} and \ref{P4} (link patterns are encoded through their Dyck word) where the last row and the last column should be ignored for the moment. From these examples it is obvious that many entries vanish. In the following proposition we provide necessary conditions for an entry to be non-zero. The proposition also implies that there exists an ordering of rows and columns such that the matrix is triangular with $1$'s along the diagonal and this implies that the matrix is invertible. 

In the proposition, we denote by $\path(\mathbf{w})$ the lattice path associated with the $01$-word $\mathbf{w}$ starting at the origin where $0$'s correspond to $(1,1)$-steps and $1$'s correspond to $(1,-1)$-steps. By abuse of notation, 
we denote by $\path(\pi)$ the path of the Dyck word corresponding to $\pi$. 

\begin{figure}
\begin{tabular}{c|ccccc||c}
& 000111 & 001011 & 001101 & 010011 & 010101 \\ \hline
000111 & 1 & 0 & 0 & 0 & 0 & $\mathbf{1}$ \\
001011 & 0 & 1 & 0 & 0 & 0 & $\mathbf{2}$ \\
001101 & 0 & 0 & 1 & 0 & 0 & $\mathbf{1}$ \\
010011 & 1 & 0 & 0 & 1 & 0 & $\mathbf{2}$ \\
010101 & 0 & 0 & 0 & 0 & 1 & $\mathbf{2}$ \\ \hline \hline
             & $\mathbf{1}$ & $\mathbf{2}$ & $\mathbf{1}$ & $\mathbf{1}$ & $\mathbf{2}$  &
\end{tabular}

\caption{\label{P3} The matrix $C(3)$}
\end{figure}

\begin{figure}
\begin{tabular}{c|cccccccccccccc||c}
 00001111 & 1 & 0 & 0 & 0 & 0 & 0 & 0 & 0 & 0 & 0 & 0 & 0 & 0 & 0 &  $\mathbf{1}$ \\
00010111 & 0 & 1  & 0 & 0 & 0 & 0 & 0 & 0 & 0 & 0 & 0 & 0 & 0 & 0  & $\mathbf{3}$ \\ 
00011011 & 0 & 0  & 1 & 0 & 0 & 0 & 0 & 0 & 0 & 0 & 0 & 0 & 0 & 0  & $\mathbf{3}$ \\ 
00100111 & 1 & 0  & 0 & 1 & 0 & 0 & 0 & 0 & 0 & 0 & 0 & 0 & 0 & 0  & $\mathbf{4}$ \\ 
00101011 & 0 & 0  & 0 & 0 & 1 & 0 & 0 & 0 & 0 & 0 & 0 & 0 & 0 & 0  & $\mathbf{7}$ \\ 
00011101 & 0 & 0  & 0 & 0 & 0 & 1 & 0 & 0 & 0 & 0 & 0 & 0 & 0 & 0  & $\mathbf{1}$\\ 
00101101 & 0 & 0  & 0 & 0 & 0 & 0 & 1 & 0 & 0 & 0 & 0 & 0 & 0 & 0  & $\mathbf{3}$ \\ 
00110011 & 0 & 0  & 0 & 1 & 0 & 0 & 0 & 1 & 0 & 0 & 0 & 0 & 0 & 0  & $\mathbf{4}$ \\ 
00110101 & 0 & 0  & 0 & 0 & 0 & 0 & 0 & 0 & 1 & 0 & 0 & 0 & 0 & 0  & $\mathbf{3}$ \\ 
01000111 & -1 & 1  & 0 & 0 & 0 & 0 & 0 & 0 & 0 & 1 & 0 & 0 & 0 & 0  & $\mathbf{3}$ \\ 
01001011 & 1 & 0 & 1 & 0 & 0 & 0 & 0 & 0 & 0 & 0 & 1 & 0 & 0 & 0 & $\mathbf{7}$ \\
01001101 &  0 & 0 & 0 & 0 & 0  & 1 & 0 & 0 & 0 & 0 & 0 & 1 & 0 & 0 & $\mathbf{4}$ \\
01010011 & 0 & 1 & 0 & 0 & 0 & 0 & 0 & 0 & 0 & 1 & 0 & 0 & 1 & 0 & $\mathbf{7}$ \\
01010101&   0 & 0 & 0 & 0 & 0 & 0 & 0 & 0 & 0 & 0 & 0 & 0 & 0  & 1 & $\mathbf{7}$ \\ \hline \hline
 & $\mathbf{1}$ &  $\mathbf{3}$ &  $\mathbf{3}$ &  $\mathbf{3}$ &  
 $\mathbf{7}$ &  $\mathbf{1}$ &  $\mathbf{3}$ &  $\mathbf{1}$ & 
 $\mathbf{3}$ &  $\mathbf{1}$ &  $\mathbf{3}$ &  $\mathbf{3}$ & 
 $\mathbf{3}$ &  $\mathbf{7}$
\end{tabular}
\caption{\label{P4} The matrix $C(4)$; column ordering equals row ordering}
\end{figure}

\begin{prop}
Let $\pi$ be a link pattern of size $n$ and $\mathbf{w}$ be a $01$-word of length $2 n$ and suppose $C^{\ext}(n)_{
\mathbf{w},\pi} \not= 0$. Let 
$1 \le i < j \le 2n$ such that $\pi$ induces a link pattern $\pi'$ on $\{i,i+1,\ldots,j\}$. Then the following two conditions must be satisfied.
\begin{enumerate}
\item The subword $\mathbf{w}[i,j]$ contains at least as many $0$'s as $1$'s.
\item Suppose $d \ge 0$ is the difference between the number of $0$'s and $1$'s in $\mathbf{w}[i,j]$. Then the path $\mathbf{path}(\mathbf{w}[i,j])-d$
lies below the path $\path(\pi')$ where ``$-d$'' indicates that the path was shifted $d$ unit steps downwards.
\end{enumerate}
\end{prop}

\begin{proof}
We start with the first assertion and use induction with respect to $j-i$. If $j=i+1$ then $\mathbf{w}_i=0$ because otherwise the arc connecting 
$i$ to $i+1$ is contradicting. Otherwise the assertion follows from the induction hypothesis if  
$i$ and $j$ are not connected by an arc in $\pi$. If $i$ and $j$ are connected by an arc the induction hypothesis implies that 
$\mathbf{w}[i+1,j-1]$ contains at  least as many $0$'s as $1$'s.
If $\mathbf{w}[i,j]$ has less $0$'s than $1$'s then  
$\mathbf{w}[i+1,j-1]$  has precisely as many $0$'s and $1$'s for otherwise $\mathbf{w}[i+1,j-1]$ contains at least 
two more $0$'s than $1$'s as the length of the subword is even.
But this implies $\mathbf{w}_i=0$, for otherwise the
arc of $\pi$ connecting $i$ and $j$ must be contradicting. 

Also for the second assertion we use induction with respect to $j-i$ and observe that it is obvious if $j=i+1$ since
$\mathbf{w}_i=0$ in this case: namely, if also $\mathbf{w}_j=0$ then $d=2$ and $\path(\mathbf{w}[i,j])-2$ is below the $x$-axis and thus below 
$\path(\pi')$, while $d=0$ and $\path(\mathbf{w}[i,j])=\path(\pi')$ if $\mathbf{w}_j=1$.

Otherwise we first assume that $i$ and $j$ are not connected by an arc and choose $k$, $i < k < j$, such that 
$\pi$ induces link patterns $\pi_1,\pi_2$ on $i,i+1,\ldots,k$ and on $k+1,k+2,\ldots,j$, respectively. Let $d_1$ and $d_2$ be the difference between the number of 
$0$'s and $1$'s in $\mathbf{w}[i,k]$ and $\mathbf{w}[k+1,j]$, respectively. By the induction hypothesis we know that 
$\path(\mathbf{w}[i,k])-d_1$ lies below $\path(\pi_1)$ and $\path(\mathbf{w}[k+1,j])-d_2$ lies below 
$\path(\pi_2)$. Now the assertion follows since the restriction of $\path(\mathbf{w}[i,j])-d$ to the first 
$k-i+1$ letters is below $\path(\mathbf{w}[i,k])-d_1$ since $d_1 \le d$, while the restriction 
of $\path(\mathbf{w}[i,j])-d$ to the last
$j-k$ letters equals $\path(\mathbf{w}[k+1,j])-d_2$ since $-d+d_1=d_2$. 

Thus we may assume that $i$ and $j$ are connected by an arc and 
denote by $\pi''$ the link pattern on $\{i+1,i+2,\ldots,j-1\}$. We have to distinguish cases according to the  
four possible choices for $(\mathbf{w}_i,\mathbf{w}_j)$. 

{\it Case} $(\mathbf{w}_i,\mathbf{w}_j)=(0,0)$: Then  
$\path(w[i+1,j-1])-d+2$ lies below $\path(\pi'')$ by the induction hypothesis which implies that 
$\path(w[i,j])-d$ is at least two units below $\path(\pi)$ except for possibly the last step.

{\it Case} $(\mathbf{w}_i,\mathbf{w}_j)=(0,1)$: Here $\path(w[i+1,j-1])-d$ lies below $\path(\pi'')$ and the assertion follows immediately.

{\it Case} $(\mathbf{w}_i,\mathbf{w}_j)=(1,0)$: The path $\path(w[i+1,j-1])-d$ lies below $\path(\pi'')$ and this implies that
$\path(w[i,j])-d$ is at least two units below $\path(\pi)$ except for possibly the first and the last step.

{\it Case}  $(\mathbf{w}_i,\mathbf{w}_j)=(1,1)$: The path $\path(w[i+1,j-1])-d-2$ lies below $\path(\pi'')$ and since $\path(w[i,j])-d$ starts with a downstep and 
$\path(\pi')$ with an upstep, the $y$-coordinates of the second lattice points of the two paths differs by $d+2$, which implies the assertion also in this case.
\end{proof}

The set of Dyck words of length $2n$ is denoted by $\mathcal{D}_n$ and the set of link patterns of 
size $n$ is denoted by $\lp(n)$.
The following was proven in \cite{QKZ}
\begin{theo} For any positive integer $n$, we have 
$$
(\psi_{\pi})_{\pi \in \lp(n)} = C(n)^{-1} (\Psi_{\alpha(\mathbf{w})})_{\mathbf{w} \in \mathcal{D}_n}.
$$
\end{theo}

To illustrate this with the help of an example, observe that the $\Psi_{\alpha(\mathbf{w})}$ are displayed in the last columns in Figures~\ref{P3} and \ref{P4}, while the $\psi_{\pi}$ are displayed in the last row in these figures. It is interesting to note that this formula for the $\psi_{\pi}$ was not ``directly'' derived from working with FPLs, but by exploring the system of linear equations for these numbers that is provided by the Razumov-Stroganov-Cantini-Sportiello theorem \cite{CanSportProof}, while the formula in Theorem~\ref{main} was directly derived from ASTs.

\section{Future work}

It seems that Theorem~\ref{rogerrefined} can be generalized in several exciting directions as was first discovered empirically again by Behrend. These generalizations will be studied in joint papers with Behrend. More concretely, additional statistics can be added. For an ASM $A$, let $T_L(A)$ be the maximal number of SW-NE all-$0$ diagonals starting from the top-left corner and let $T_R(A)$ be the maximal number of NW-SE all-$0$ diagonals starting from the top-right corner. (In terms of monotone triangles, ASMs with prescribed statistic $T_L$ correspond to Gog trapezoids, while ASMs with prescribed statistics $T_L$ and $T_R$ correspond to Gog pentagons.) For an AST $A$, let $T_L(A)$ be the maximal number of all-$0$ columns starting from the left end of $A$ and let $T_R(A)$ be the maximal number of all $0$-columns starting from the right end of $A$.  Moreover, $\mu(A)$ denotes the number of $-1$'s, while $\inv(A)$ denotes the inversion number if $A$ is an ASM or an AST. Then the joint distributions of $T_L,T_R,\mu,\inv$ are the same on the sets of order $n$ ASMs and order $n$ ASTs. This generalizes Theorem 5.5 from \cite{EXTREMEDASASM}. Second, the joint distribution of $T_L$ and $\rho$ on the set of order $n$ ASTs is the same as the joint distribution of $T_L$ and the column of the $1$ in the top row of the ASM on the set of order $n$ ASMs, which generalizes Theorem~\ref{rogerrefined}.

\section{Acknowledgement} 

I thank Roger Behrend for helpful discussions and sharing the above mentioned conjectures with me.

\end{document}